\documentclass{article}
\usepackage{amssymb,latexsym,amsfonts,amsmath,amsthm}
\usepackage{graphics,graphicx,epstopdf}
\usepackage{enumitem}
\usepackage[colorlinks=true,linkcolor=black,citecolor=blue,linktocpage=true]{hyperref}

\newtheorem{theorem}{Theorem}[section]
\newtheorem{corollary}[theorem]{Corollary}
\newtheorem{lemma}[theorem]{Lemma}
\newtheorem{definition}{Definition}[section]

\newtheorem*{remark}{Remark}

\newcommand{\N}{\mathbb{N}}

\newcommand{\R}{\mathbb{R}}

\newcommand{\dual}[1]{ {#1}^{\ast} }
\newcommand{\ddual}[1]{ {#1}^{\ast\ast} }

\newcommand{\ap}[1]{ \textnormal{ap}\left(#1\right) }
\newcommand{\wap}[1]{ \textnormal{wap}\left(#1\right) }
\newcommand{\cc}[1]{ \textnormal{cc}\left(#1\right) }
\newcommand{\rcc}[1]{ \textnormal{rcc}\left(#1\right) }
\newcommand{\lcc}[1]{ \textnormal{lcc}\left(#1\right) }
\newcommand{\wpL}[1]{ \textnormal{wpL}\left(#1\right) }

\begin{document}
\title{Weak Sequential Properties of the Multiplication Operators on Banach Algebras}
\author{Onur Oktay}
\maketitle

\begin{abstract}
Let $A$ be a Banach algebra. For $f\in\dual{A}$, we inspect the weak sequential properties of the well-known map $T_f:A\to\dual{A}$, $T_f(a) = fa$, where $fa\in\dual{A}$ is defined by $fa(x) = f(ax)$ for all $x\in A$. We provide equivalent conditions for when $T_f$ is completely continuous for every  $f\in\dual{A}$, and for when $T_f$ maps weakly precompact sets onto L-sets for every $f\in\dual{A}$. Our results have applications to the algebra of compact operators $K(X)$ on a Banach space $X$.
\end{abstract}


\section{Introduction}

It was conjectured in \cite{GRW92} that reflexive amenable Banach algebras are finite dimensional. 
Runde proved this conjecture with additional hypotheses. Precisely, every reflexive amenable Banach algebra $A$ with the approximation property, where the set of almost periodic functionals $\ap{A}$ separates the points, is finite dimensional, see e.g., Corollary 3.3 and the preceding Remark in \cite{Runde01-2}.

$\ap{A}$ is related to the joint weak continuity properties of the multiplication \cite{DuncanHosseiniun79}. When $A$ is reflexive, $\ap{A}$ is separating if and only if the multiplication is jointly weakly sequentially continuous if and only if the multiplication operator $T_f$ is completely continuous for every $f\in\dual{A}$. From this perspective, it is worthwhile to study the weak sequential properties of the multiplication and the associated operators for a general Banach algebra.

In Section~\ref{sec:wp2L}, we study the joint weak sequential continuity of the Banach algebra multiplication, and the associated multiplication operators. In analogy to $\ap{A}$, we define the subspaces $\wpL{A}$, $\lcc{A}$ and $\rcc{A}$ of $\dual{A}$. We provide characterizations of the joint weak sequential continuity of the multiplication in terms of those subspaces. We inspect the relationship between these functionals and several Banach space properties of $A$. The tools we develop in Section~\ref{sec:wp2L} have applications to the algebra of compact operators $K(X)$ on a Banach space $X$, which is contained in Section~\ref{sec:KX}.


\section{Notation and Preliminaries}\label{sec:prelim}

Let $A$ be a Banach space. 
$A$ has {\it Schur property} if every weakly convergent sequence in $A$ is norm convergent. 
$A$ has {\it Dunford-Pettis property (DPP)} if $f_n(x_n)\to 0$ for all weakly null sequences $(x_n)$ in $A$ and $(f_n)$ in $\dual{A}$. 
$A$ has {\it Reciprocal Dunford-Pettis property (RDPP)} if every completely continuous operator from $A$ into another Banach space is weakly compact.

$E\subseteq\dual{A}$ is an {\it L-set} if for every weakly null sequences $(x_n)$ in $A$
$$\lim_{n\to\infty}\sup_{f\in E} |f(x_n)| = 0.$$
It is not difficult to show that $E$ is an L-set if and only if $f_n(x_n)\to 0$ for every sequence $(f_n)$ in $E$, and every weakly null sequence $(x_n)$ in $A$.
Clearly, $A$ has Schur property if and only if the unit ball of $\dual{A}$ is an L-set. 
$A$ has DPP if and only if every weakly precompact subset of $\dual{A}$ is an L-set (\cite{Graves80-ch2}, p.18). 
$A$ has RDPP if and only if every L-set contained in $\dual{A}$ is relatively weakly compact \cite{Leavelle84},\cite{Emmanuele91}. 
$A$ contains no copy of $\ell^1$ if and only if every L-set contained in $\dual{A}$ is relatively compact \cite{Emmanuele86}.


\vspace{3mm}

When $A$ is a Banach algebra, one defines the left and right multiplication operators $L_a,R_a:A\to A$ by $L_ax = ax$ and $R_ax = xa$. Moreover,
%
\begin{eqnarray*}
fa(x) = f(ax) \hspace{7mm} \mu f(a) = \mu(fa) \hspace{7mm} (\mu\cdot_1\nu)(f) = \mu(\nu f) \\
af(x) = f(xa) \hspace{7mm} f\mu (a) = \mu(af) \hspace{7mm} (\nu\cdot_2\mu)(f) = \mu(f\nu)
\end{eqnarray*}
for all $a,x\in A$, $f\in\dual{A}$ and $\mu,\nu\in\ddual{A}$.
The products in the first column make $\dual{A}$ a two-sided $A$-module, and naturally define two multiplication operators $S_f, T_f:A\to\dual{A}$ given by $T_fa = fa$ and $S_fa = af$, for each $f\in\dual{A}$ and all $a\in A$. 
$\ap{A}$ (resp. $\wap{A}$) denotes the space of all $f\in\dual{A}$ for which $T_f$ is a (weakly) compact operator. By duality, $T_f$ is (weakly) compact if and only if $S_f$ is (weakly) compact.

$\ddual{A}$ is a Banach algebra with each of the {\it Arens products} $\cdot_1$ and $\cdot_2$.
The canonical injection of $A$ into $\ddual{A}$ is an isometric isomorphism onto its range for both products. $A$ is {\it Arens regular} if both Arens products are equal on $\ddual{A}$. It is well-known that $A$ is Arens regular if and only if $\dual{A} = \wap{A}$ \cite{Palmer1}. Also $\dual{A} = \wap{A}$ (resp. $\dual{A} = \ap{A}$) if and only if one (thus both) of the Arens products is separately (resp. jointly) bounded weak* continuous \cite{DuncanHosseiniun79}.


\section{Weak sequential continuity of multiplication}\label{sec:wp2L}

Analogous to $\wap{A}$ and $\ap{A}$, we define the subspaces $\wpL{A}$, $\rcc{A}$, $\lcc{A}$ and $\cc{A}$ as below.

\begin{definition}\label{def:wp2L}\normalfont
Let $A$ be a Banach algebra.
\begin{enumerate}[label=\roman*.]
\item $\wpL{A}$ is the set of all $f\in\dual{A}$ such that $T_f$ maps weakly precompact sets onto L-sets,
\item $\lcc{A}$ (resp. $\rcc{A}$) is the set of all $f\in\dual{A}$ such that $T_f$ (resp. $S_f$) is completely continuous,
\end{enumerate}
and $\cc{A} = \lcc{A}\cap\rcc{A}$.
\end{definition}
Clearly $\lcc{A} = \rcc{A}$ if $A$ is commutative. 
Also $\ap{A}\subseteq\cc{A}$, and 
$\lcc{A}\cup\rcc{A}\subseteq\wpL{A}$ by Lemmas~\ref{lemma:wp2L} and \ref{lemma:cc} below.

\begin{lemma}\normalfont\label{lemma:wp2L}
Let $A$ be a Banach algebra and $f\in\dual{A}$. Then, the following are equivalent.
\begin{enumerate}
\item[i.] If $(x_n),(y_n)$ are weakly null sequences in $A$, then $f(x_ny_n)\to 0$.
\item[ii.] If $(x_n)$ is a weakly null and $(y_n)$ is a weakly Cauchy sequence in $A$, then $f(x_ny_n)\to 0$.
\item[iii.] $T_f$ map weakly precompact sets onto L-sets.
\item[iii'.] $S_f$ map weakly precompact sets onto L-sets.
\end{enumerate}
\end{lemma}
\begin{proof}
$(i\Rightarrow ii)$ Suppose $f(x_ny_n)\not\to 0$. WLOG, there exist $r>0$ such that $|f(x_ny_n)|\geq r$ for all $n\in\N$.
Since $(fx_n)$ is a weakly null sequence in $\dual{A}$, then $f(x_{n_k}y_k) \to 0$ for a subsequence $(fx_{n_k})_{k\in\N}$. Since $(y_{n_k}-y_k)_{k\in\N}$ is weakly null, then $$\lim_{k\to\infty} f(x_{n_k}y_{n_k}) = \lim_{k\to\infty} f(x_{n_k}(y_{n_k}-y_k)) + f(x_{n_k}y_k) = 0.$$ Contradiction.

$(ii\Rightarrow iii)$ Suppose there exists a weakly precompact $K\subseteq A$ such that $T_f(K)$ is not an L-set. Then, there exists a $r>0$ and a weakly null sequence $(y_n)$ in $A$ such that $\sup_{x\in K}|fx(y_n)|\geq 2r$ for all $n\in\N$. Further, there exists a sequence $(x_n)$ in $K$ such that $|f(x_ny_n)|\geq r$ for all $n\in\N$. Since $K$ is weakly precompact, there is a weakly Cauchy subsequence $(x_{n_k})$. Hence, $f(x_{n_k}y_{n_k})\not\to 0$ while $(x_{n_k})$ is weakly Cauchy and $(y_{n_k})$ weakly null. Contradiction.

Assuming that $S_f$ does not map weakly precompact sets onto L-sets similarly leads to a contradiction.

$(iii\Rightarrow i)$ Let $(x_n),(y_n)$ be weakly null. $\{fx_n:n\in\N\}$ is an L-set by hypothesis, so $f(x_ny_n)\to 0$.
\end{proof}

\begin{lemma}\normalfont\label{lemma:cc}
Let $A$ be a Banach algebra and $f\in\dual{A}$. Then, the following are equivalent.
\begin{enumerate}
\item[i.]  If $(x_n)$ is weakly null and $(y_n)$ is bounded, then $f(x_ny_n)\to 0$ (resp. $f(y_nx_n)\to 0$)
\item[ii.] $T_f$ (resp. $S_f$) is completely continuous.
\item[iii.] $S_f$ (resp. $T_f$) maps bounded sets onto L-sets.
\end{enumerate}
\end{lemma}
\begin{proof}
$(i\Rightarrow ii)$ Suppose $T_f$ is not completely continuous. Then, there exists a weakly null $(x_n)$ and $r>0$ such that $\|fx_n\|\geq 2r$ for all $n\in\N$. For each $n\in\N$, there exists $y_n\in A$, $\|y_n\|=1$ such that $|f(x_ny_n)|\geq r$. Hence, $f(x_ny_n)\not\to 0$.

$(ii\Rightarrow i)$ Since $\|fx_n\|\to 0$, then $|f(x_ny_n)|\leq \|fx_n\|\|y_n\|\to 0$.

$(i\Rightarrow iii)$ Suppose $S_f$ does not map bounded sets onto L-sets. Then, there exists a bounded $(y_n)_{n\in\N}$, a weakly null $(x_n)$ and $r>0$ such that $|y_nf(x_n)|\geq r$ for all $n\in\N$. Thus, $f(x_ny_n)\not\to 0$.

$(iii\Rightarrow i)$ Since $\{y_nf: n\in\N\}$ is an L-set by hypothesis, then \hbox{$f(x_ny_n) = y_nf(x_n)\to 0$.}
\end{proof}

\begin{lemma}\label{lemma:wpL_not_cc}\normalfont
Let $A$ be a Banach algebra and $f\in\wpL{A}$. The following are equivalent.
\begin{enumerate}
\item[i.] $f\notin\lcc{A}$ (resp. $f\notin\rcc{A}$). 
\item[ii.] $A$ contains an $\ell^1$ sequence $(w_n)$ and a weakly null sequence $(u_n)$ such that $f(u_mw_n) = \delta_{mn}$ (resp. $f(w_nu_m) = \delta_{mn}$) for all $m\geq n$.
\end{enumerate}
\end{lemma}
\begin{proof}
$(ii\Rightarrow i)$ Clear, since $E=\{w_n:n\in\N\}$ is a bounded set such that $T_f(E) = \{fw_n:n\in\N\}$ (resp. $S_f(E) = \{ w_nf :n\in\N\}$ ) is not an L-set. Thus, $f\notin\lcc{A}$ (resp. $f\notin\rcc{A}$).

$(i\Rightarrow ii)$ If $f\notin\lcc{A}$ (resp. $f\notin\rcc{A}$), then 
there exist $r>0$, a weakly null sequence $(x_n)$ and a bounded sequence $(y_n)$ such that $|f(x_ny_n)|\geq r$ (resp. $|f(y_nx_n)|\geq r$) by Lemma~\ref{lemma:cc}. $(y_n)$ cannot have a weakly Cauchy subsequence by Lemma~\ref{lemma:wp2L}, thus it has an $\ell^1$ subsequence by Rosenthal's theorem. Without loss of generality, we may assume $(y_n)$ is an $\ell^1$ sequence. Further, replacing $y_n$ by $\tilde{y_n} = y_n/\|y_n\|$ and $x_n$ by $\tilde{x}_n=|f(x_ny_n)|^{-1}\|y_n\| x_n$ (resp. $\tilde{x}_n = |f(x_ny_n)|^{-1}\|y_n\| x_n$), we may assume that $f(x_ny_n)=1$ (resp. $f(y_nx_n)=1$) for a weakly null sequence $(x_n)$ and a unit norm $\ell^1$ sequence $(y_n)$.


First, assume $f\notin\lcc{A}$ and $\|f\|\leq 1$. 
Let $c>0$ be a bound for the sequence $(\|x_n\|)$.
Proof is by induction. 
For the base step, let $u_1 = x_1$ and $w_1 = y_1$.
For the inductive step, assume that, for $k,l=1,\dots,n-1$,
there exist $x_{m_k}, u_k$ such that 
$\|x_{m_k} - u_k\|\leq (1+c)^{-k}$,
$f(u_ly_{n_k}) = \delta_{kl}$ for $l\geq k$. 
Here, $w_k=y_{m_k}$.

We will construct $u_n$ and $w_n$. Since $(x_n)$ is weakly null, there exists $m_n\geq m_{n-1}$ such that 
$$\max\{|f(x_{m_n}w_k)|: k=1,\dots,n-1 \} \leq \frac{(1+c)^{-(n+2)}}{1+ (1+c)^{-(n+1)}}.$$
There exists $\tilde{u}_n\in\bigcap_{k=1}^{n-1} ker(w_kf)$ such that 
$$\|x_{m_n}-\tilde{u}_n\|\leq \frac{(1+c)^{-(n+1)}}{ 1+ (1+c)^{-(n+1)}}.$$
Let $w_n = y_{m_n}$ and $u_n = \tilde{u}_n /f(\tilde{u}_nw_n)$.
Then, $f(u_nw_n) = 1$, $f(u_nw_k) = 0$ for $k=1,\dots,n-1$, and 
\begin{eqnarray*}
\|x_{m_n}-u_n\|
\leq \|x_{m_n} - \tilde{u}_n\| + \|\tilde{u}_n\| \frac{|1-f(\tilde{u}_nw_n)|}{|f(\tilde{u}_nw_n)|} 
\leq (1+c)\frac{\|x_{m_n} - \tilde{u}_n\|}{1-\|x_{m_n} - \tilde{u}_n\|} 
\leq (1+c)^{-n}.
\end{eqnarray*}
Consequently, $(u_n)$ is weakly null. 
Clearly, being a subsequence of an $\ell^1$ sequence, $(w_n)$ is an $\ell^1$ sequence. 

The other case, $f\notin\rcc{A}$, is handled similarly.
\end{proof}

\begin{remark}\normalfont
If not only the sequence $(u_n)$ in Lemma~\ref{lemma:wpL_not_cc} is weakly null, but also $\sum u_n$ is a weakly unconditionally convergent series, then $(w_n)$ is a complemented $\ell^1$ sequence. 
In fact, let $W$ be the closed linear span of $(w_n)$, which is isomorphic to $\ell^1$. 
Let $S:A\to W$ be defined by $Sx = \sum_{n\in\N} f(u_nx)w_n$ (resp. $Sx = \sum_{n\in\N} f(xu_n)w_n$). Then, $S$ is a surjective bounded linear operator, and thus, $A$ contains a complemented copy of $\ell^1$, e.g., by Theorem 5 in page 72 of \cite{gtm092}.

\end{remark}

\begin{lemma}\label{lemma:TfLa}\normalfont
Let $A$ be a Banach algebra, $f\in\dual{A}$ and $a,b\in A$.
\begin{enumerate}[label=\alph*.]
\item If $f\in\wpL{A}$ and $L_a$ (resp. $R_a$) is a weakly precompact operator, then $fa\in\rcc{A}$ (resp. $af\in\lcc{A}$).

\item If $f\in\lcc{A}$ (resp. $f\in\rcc{A}$) and $L_a$ (resp. $R_a$) is weakly precompact, then $fa\in\ap{A}$ (resp. $af\in\ap{A}$).

\item If $f\in\wpL{A}$ and $L_a$, $R_b$ are weakly precompact operators, then $bfa\in\ap{A}$.
\end{enumerate}
\end{lemma}
\begin{proof}
First, $T_{fa} = T_fL_a$ and $S_{af} = S_fR_a$.

a. If $B$ is a bounded subset, then $T_{fa}(B) = T_f(L_a(B))$ is an L-set. Thus, $fa\in\rcc{A}$ by Lemma~\ref{lemma:cc}. 
Similarly, $S_{af}$ maps bounded sets onto L-sets, so $af\in\lcc{A}$ by Lemma~\ref{lemma:cc}.

b. If $B$ is a bounded subset, then $T_{fa}(B) = T_f(L_a(B))$ is relatively compact, so $fa\in\ap{A}$. 
Similarly, $S_{af} = S_fR_a$ is a compact operator, i.e., $af\in\ap{A}$.

c. a direct consequence of (a) and (b).
\end{proof}

\begin{theorem}\label{thm:no-l1-wp2L}\normalfont
Let $A$ be a Banach algebra.
If $A$ contains no copy of $\ell^1$, then $\wpL{A} = \cc{A} = \ap{A}$. 
\end{theorem}
\begin{proof}
$\wpL{A} = \cc{A}$ by Lemma~\ref{lemma:wpL_not_cc}.
Second, let $f\in\cc{A}$. Every completely continuous operator from $A$ is compact when $A$ does not contain a copy of $\ell^1$. Thus, $T_f$ is compact, i.e., $f\in\ap{A}$.
%
%
%
\end{proof}

\begin{definition}\label{def:jwsc}\normalfont
Let $A$ be a Banach algebra. The multiplication of $A$ is
{\it jointly weakly sequentially continuous (jwsc)} if whenever $(x_n),(y_n)$ are weakly null sequences, then so is $(x_ny_n)$. 

We say that the multiplication is {\it l-strong jwsc} (resp. {\it r-strong jwsc}) if $(x_ny_n)$ (resp. $(y_nx_n)$) is weakly null whenever $(x_n)$ is a weakly null sequence, and $(y_n)$ is a bounded sequence. The multiplication is {\it strong jwsc} if it is both l-strong and r-strong jwsc.
\end{definition}
The space of continuous functions $C(K)$ on a compact Hausdorff topological space is an example of a Banach algebra with strong jwsc multiplication. The group algebra $L^1(\R)$ is an example with jwsc multiplication, which is neither l-strong jwsc nor r-strong jwsc.

\begin{theorem}\label{thm1}\normalfont
Let $A$ be a Banach algebra. Then, 
\begin{enumerate}[label=\alph*.]
\item $\dual{A} = \wpL{A}$ if and only if $A$ has jwsc multiplication.
\item $\dual{A} = \lcc{A}$ (resp. $\dual{A} = \rcc{A}$) if and only if $A$ has l-strong (resp. r-strong) jwsc multiplication.
\item If $A$ has DPP, then $\dual{A} = \wpL{A}$. 
If $A$ is also Arens regular, then $\dual{A} = \cc{A}$.
\item If $A$ contains no copy of $\ell^1$ and has jwsc multiplication, then $\dual{A} = \ap{A}$.
\item If $\dual{A}$ has Schur property, then $\dual{A} = \ap{A}$. 
\item If $A$ has Schur property, then $\dual{A}=\cc{A}$.
\end{enumerate}
\end{theorem}
\begin{proof}
a. Immediate by Lemma~\ref{lemma:wp2L}. \hspace{5mm}
b. Immediate by Lemma~\ref{lemma:cc}.

c. It is well-known that if $A$ has DPP, then $A$ has jwsc multiplication, see e.g., Proposition 2.34 in \cite{Dineen99}. In fact, let $f\in\dual{A}$ and $(x_n),(y_n)$ be two weakly null sequences in $A$. Since $(fx_n)$ is a weakly null sequence in $\dual{A}$, then $f(x_ny_n) = fx_n(y_n)\to 0$ by the DPP. Thus, $\dual{A} = \wpL{A}$ by (a).

Second, $\dual{A}=\wap{A}$ by Arens regularity and $\wap{A}\subseteq\cc{A}$ by DPP. Thus, $\dual{A} = \cc{A}$.

d. By (a) and Theorem~\ref{thm:no-l1-wp2L}.

e. $\dual{A}$ has Schur property if and only if $A$ has DPP and contains no copy of $\ell^1$ (see Theorem 3 in \cite{Graves80-ch2}). Thus, $\ap{A} = \dual{A}$ by (c) and (d).

f. Clearly, every bounded linear operator defined on $A$ is completely continuous.
\end{proof}

It is worth to note that there are commutative Banach algebras, which have DPP, but $\dual{A}\neq\cc{A}$. In fact, if $G$ is a non-discrete non-compact locally compact Abelian group, then the group algebra $L^1(G)$ has the DPP, and there exists $g\in C_b(G)$ for which the map $T_g:L^1(G)\to L^{\infty}(G)$, $T_g(x) = g*x$ is not completely continuous \cite{CrombezGovaerts84}. 

On the other hand, there are commutative Banach algebras, which does not have DPP, but $\dual{A}=\ap{A}$. In fact, $\ell^p$ ($1<p<\infty$) is a Banach algebra with pointwise operations. 
It is not difficult to show that the pointwise multiplication is jwsc. Being reflexive, $\ell^p$ does not have DPP and contains no copy of $\ell^1$. Thus, $\dual{A}=\ap{A}$ by (d) of Theorem~\ref{thm1}.

It is clear that (e) of Theorem~\ref{thm1} does not have a converse for $\ell^p$. However, if $A$ is a C*-algebra, then $\dual{A}=\ap{A}$ if and only if $\dual{A}$ have Schur property by Theorem 3.6 in \cite{UlgerLau93}. 
We provide an analogous statement below.
\begin{corollary}\label{thm:Cstar}\normalfont
Let $A$ be a C*-algebra. Then, the following are equivalent.
\begin{enumerate}
\item[i.] $\dual{A} = \cc{A}$. 
\item[ii.] $\dual{A} = \wpL{A}$.
\item[iii.] $A$ has DPP.
\end{enumerate}
\end{corollary}
\begin{proof}
%
$(i\Rightarrow ii)$ clear since $\cc{A}\subseteq\wpL{A}$.

$(ii\Rightarrow iii)$ if $\dual{A}=\wpL{A}$, then $A$ has jwsc multiplication by Theorem~\ref{thm1}. 
Particularly, $(x_nx_n^{*})$ is a weakly null sequence whenever $(x_n)$ is weakly null. Equivalently, $A$ has DPP by Theorem 1 in \cite{ChuIochum90}.

$(iii\Rightarrow i)$ C*-algebras are Arens regular, so $\dual{A}=\wap{A}$. 
Every weakly compact operator defined on $A$ is completely continuous by DPP. Thus, $\wap{A}\subseteq\cc{A}$ and so $\dual{A}=\cc{A}$.
\end{proof}

\begin{theorem}\label{thm:jwsc-idealinbidual}\normalfont
Let $A$ be a Banach algebra, which is a left (resp. right) ideal in $\ddual{A}$. If the multiplication is jwsc, then 
\begin{enumerate}[label=\alph*.]
\item $\dual{A}A\subseteq\rcc{A}$ (resp. $A\dual{A}\subseteq\lcc{A}$).
\item $A\dual{A}A\subseteq\ap{A}$.
\end{enumerate}
\end{theorem}
\begin{proof}
It is well-known that $A$ is a left (resp. right) ideal in $\ddual{A}$ if and only if $L_a$ (resp. $R_a$) is weakly compact for all $a\in A$, see e.g., \cite{Palmer1}. 
Second, $\dual{A}=\wpL{A}$ by Theorem~\ref{thm1}. 

a. $fa\in\rcc{A}$ (resp. $af\in\lcc{A}$) for all $f\in\dual{A}$ and $a\in A$ by Lemma~\ref{lemma:TfLa}. Hence, $\dual{A}A\subseteq\rcc{A}$ (resp. $A\dual{A}\subseteq\lcc{A}$). 

b. $bfa\in\ap{A}$ for all $f\in\dual{A}$ and $a,b\in A$ by Lemma~\ref{lemma:TfLa}. Hence, the result.
\end{proof}

It is clear from the proof that the result of Theorem~\ref{thm:jwsc-idealinbidual} persists if we assumed that $L_a$ (resp. $R_a$) is weakly precompact for all $a\in A$, instead of $A\ddual{A}\subseteq A$ (resp. $\ddual{A}A\subseteq A$). The same is valid for the next corollary.

\begin{corollary}\normalfont 
Let $A$ be a Banach algebra 
with a bounded left (resp. right) approximate identity, 
which is a right (resp. left) ideal in $\ddual{A}$. 
If the multiplication of $A$ is jwsc, then $\wap{A}\subseteq\lcc{A}$ (resp. $\wap{A}\subseteq\rcc{A}$).
If, additionally, $A$ is Arens regular, then $\dual{A}=\lcc{A}$ (resp. $\dual{A}=\rcc{A}$).
\end{corollary}
\begin{proof}
$\wap{A}\subseteq A\dual{A}$ (resp. $\wap{A}\subseteq\dual{A}A$) by Theorem 3.1 in \cite{Ulger90}. Thus, if the multiplication is jwsc, then $\wap{A}\subseteq\lcc{A}$ (resp. $\wap{A}\subseteq\rcc{A}$) by Theorem~\ref{thm:jwsc-idealinbidual}. The second result is obvious since $\dual{A}=\wap{A}$ if $A$ is Arens regular.
\end{proof}


\section{The Algebra of Compact Operators $K(X)$}\label{sec:KX}

Let $X$ be a Banach space. $K(X)$ denotes the Banach algebra of compact operators $X\to X$.
The following theorem summarizes some of the well-known results about $K(X)$, which we will need subsequently.
\begin{theorem}\normalfont\label{thm:KX-summary}
Suppose $X$ is an infinite dimensional Banach space and let $A=K(X)$. Then,
\begin{enumerate}[label=\alph*.]
\item $\wap{A}=\dual{A}$ if and only if $X$ is reflexive. \label{itm:KX_Arens}
\item $\wpL{A}=\dual{A}$ if and only if $X$ has DPP. \label{itm:KX_jwsc}
\item If $X$ has the approximation property, then $\ap{A}=\{0\}$. \label{itm:KX_ap}
\item $A$ cannot both be Arens regular and have jwsc multiplication. \label{itm:KX_Arens_jwsc}
\item If $X$ is reflexive, then $A$ contains no copy of $\ell^1$. \label{itm:KX_l1-0}
\item If $\wpL{A}=\dual{A}$, then $A$ contains a copy of $\ell^1$. \label{itm:KX_l1-1}
\end{enumerate}
\end{theorem}
\begin{proof}
\ref{itm:KX_Arens} Theorem 3 in \cite{Young76}.
\hspace{5mm}
\ref{itm:KX_jwsc} \cite{Asthagiri83} and Theorem~\ref{thm1}.

\ref{itm:KX_ap} Proposition 3.3 in \cite{UlgerDuncan92}.

\ref{itm:KX_Arens_jwsc} by \eqref{itm:KX_Arens}, \eqref{itm:KX_jwsc}, and the fact that every reflexive Banach space with the DPP is finite dimensional.

\ref{itm:KX_l1-0} $\dual{A}$ has Radon-Nikodym property by \cite{Ruess84}. Equivalently, every separable subspace of $A$ has a separable dual \cite[p.198]{DiestelUhl77}. Hence, $A$ cannot contain a copy of $\ell^1$.

\ref{itm:KX_l1-1} $X$ has DPP by \eqref{itm:KX_jwsc}. Thus, either $X$ contains a copy of $\ell^1$, or $\dual{X}$ has Schur property. Also, every infinite dimensional Banach space with Schur property contains a copy of $\ell^1$. Hence, either $X$ or $\dual{X}$, both of which are isometrically isomorphic to complemented subspaces of $A$, contains a copy of $\ell^1$. 
\end{proof}

\begin{theorem}\normalfont\label{thm:multipliers}
Let $X$ be an infinite dimensional Banach space with the approximation property and $A=K(X)$. 
\begin{enumerate}[label=\alph*.]
\item If $\wpL{A}$ separates the points of $A$, then either $L_a$ fixes a copy of $\ell^1$ for every nonzero $a\in A$, or $R_a$ fixes a copy of $\ell^1$ for every nonzero $a\in A$.
\item If $\lcc{A}$ (resp. $\rcc{A}$) separates the points of $A$, then $L_a$ (resp. $R_a$) fixes a copy of $\ell^1$ for every nonzero $a\in A$.
\end{enumerate}
\end{theorem}
\begin{proof}
A bounded linear operator between two Banach spaces is weakly precompact if and only if it does not fix a copy of $\ell^1$. 

a. Assume, for a contradiction, that $L_a$ and $R_b$ are weakly precompact operators for two nonzero $a,b\in A$. Then, $bfa\in\ap{A}$ for all $f\in\wpL{A}$ by Lemma~\ref{lemma:TfLa}. But, $\ap{A}=\{0\}$ by Theorem~\ref{thm:KX-summary}.\ref{itm:KX_ap} Thus, $bfa = 0$ for all $f\in\wpL{A}$. Since $\wpL{A}$ is separating, then $L_aR_b=0$. Hence, either $a=0$ or $b=0$. Contradiction.

b. If $R_a$ (resp. $L_a$) is weakly precompact, then $af\in\ap{A}$ (resp. $fa\in\ap{A}$) for all $f\in\rcc{A}$ (resp. $f\in\lcc{A}$). Since $\ap{A}=\{0\}$ and $\rcc{A}$ (resp. $\lcc{A}$) is a separating set, then $R_a=0$ (resp. $L_a=0$) so $a=0$.
\end{proof}

Theorem~\ref{thm:KX-cc} is analogous to Asthagiri's theorem in \cite{Asthagiri83}. Perhaps known by the experts in the field, Theorem~\ref{thm:KX-cc} never appeared in the literature to the best of our knowledge. The proof is almost verbatim similar to Asthagiri's proof in \cite{Asthagiri83}, nonetheless the key differences are not entirely obvious. Thus, we provide a separate proof beolow.

\begin{theorem}\normalfont\label{thm:KX-cc}
Let $X$ be a Banach space and $A=K(X)$.
\begin{enumerate}
\item[a.] $\dual{A}=\rcc{A}$ if and only if $X$ has Schur property.
\item[b.] $\dual{A}=\lcc{A}$ if and only if $\dual{X}$ has Schur property.
\item[c.] $\dual{A}\neq\cc{A}$ for any infinite dimensional Banach space $X$.
\end{enumerate}
\end{theorem}
\begin{proof}
a. First assume $\dual{A}=\rcc{A}$. Let $(x_n)$ be a weakly null sequence in $X$. It is sufficient to show that $f_n(x_n)\to 0$ for any bounded sequence $(f_n)$ in $\dual{X}$.

$(x_1\otimes f_n)$ is a bounded, $(x_n\otimes f_1)$ is a weakly null sequence in $K(X)$. Thus, $f_n(x_n) x_1\otimes f_1 = (x_1\otimes f_n)(x_n\otimes f_1) \to 0$ weakly in $K(X)$. Hence, $f_n(x_n)\to 0$.

Conversely, suppose $(x_n)$ is a bounded and $(y_n)$ is a weakly null sequence in $K(X)$. 
For each $\nu\in\ddual{X}$, $(y_n^{**}\nu)$ is a weak* null sequence in $\ddual{X}$ by \cite{Kalton74}. On the other hand, $y_n^{**}\nu\in X$ since $y_n$ are compact. Thus, $(y_n^{**}\nu)$ is weakly null in $X$, and so norm null by the Schur property.

Hence, for each $f\in\dual{X}$ and $\nu\in\ddual{X}$, 
$$|\nu( (x_ny_n)^{*}f )| = |y_n^{**}\nu( x_n^{*}f )| \leq \|y_n^{**}\nu\| \|x_n^{*}f\| \to 0, $$
i.e., $x_ny_n\to 0$ weakly in $K(X)$. Thus, $\dual{A}=\rcc{A}$ by Lemma~\ref{lemma:cc}.

b. Assume $\dual{A}=\lcc{A}$ and let $(f_n)$ be a weakly null sequence in $\dual{X}$. It is sufficient to show that $f_n(x_n)\to 0$ for any bounded sequence $(x_n)$ in $X$.

$(x_1\otimes f_n)$ is a weakly null, $(x_n\otimes f_1)$ is a bounded sequence in $K(X)$. Thus, $f_n(x_n) x_1\otimes f_1 = (x_1\otimes f_n)(x_n\otimes f_1) \to 0$ weakly in $K(X)$. Hence, $f_n(x_n)\to 0$.

Conversely, suppose $(x_n)$ is a weakly null and $(y_n)$ is a bounded sequence in $K(X)$. 
For each $f\in\dual{X}$, $(x_n^{*}f)$ is weakly null in $\dual{X}$ by \cite{Kalton74}, and so norm null by the Schur property. Hence, for each $f\in\dual{X}$ and $\nu\in\ddual{X}$, 
$$|\nu( (x_ny_n)^{*}f )| = |y_n^{**}\nu( x_n^{*}f )| \leq \|y_n^{**}\nu\| \|x_n^{*}f\| \to 0, $$
i.e., $x_ny_n\to 0$ weakly in $K(X)$. Thus, $\dual{A}=\lcc{A}$ by Lemma~\ref{lemma:cc}.

c. $X$ and $\dual{X}$ cannot both have Schur property unless $X$ is finite dimensional.
\end{proof}

An immediate consequence of Theorem~\ref{thm:KX-cc} is given below.
\begin{corollary}\normalfont\label{cor:multipliers}
Let $X$ be an infinite dimensional Banach space with the approximation property and $A=K(X)$. 
\begin{enumerate}[label=\alph*.]
\item If $X$ has DPP, then either 
$L_a$ fixes a copy of $\ell^1$ for every nonzero $a\in A$, or 
$R_a$ fixes a copy of $\ell^1$ for every nonzero $a\in A$.
\item If $X$ (resp. $\dual{X}$) has Schur property, then 
$R_a$ (resp. $L_a$) fixes a copy of $\ell^1$ for every nonzero $a\in A$.
\end{enumerate}
\end{corollary}
\begin{proof}
a. by Theorems~\ref{thm:multipliers} and \ref{thm:KX-summary}.\ref{itm:KX_jwsc}
\hspace{8mm}
b. by Theorems~\ref{thm:multipliers} and \ref{thm:KX-cc}.
\end{proof}

\bibliographystyle{amsplain}
\bibliography{/media/o2/SDA2/O/TeX/Bibtex_files/All.bib,/media/o2/SDA2/O/TeX/Bibtex_files/Textbook.bib,/media/o2/SDA2/O/TeX/Bibtex_files/Banach_algebra.bib,/home/o2/Downloads/jwsc/jwsc.bib}

\end{document}